\documentclass[english]{article} 
\usepackage[utf8]{inputenc}
\usepackage[T1]{fontenc}
\usepackage{lmodern}
\usepackage[a4paper]{geometry}
\usepackage{fullpage}

\usepackage{babel}

\usepackage{amssymb,amsmath,mathtools,amsthm,dsfont,bm,xcolor,hyperref,xurl}

\usepackage[all]{xy}

\newtheorem{theorem}{Theorem}
\newtheorem{prop}[theorem]{Proposition}

\newtheorem{lemma}[theorem]{Lemma}

\theoremstyle{definition}

\newcommand{\E}{\mathbb{E}}

\renewcommand{\P}{\mathbb{P}}

\newcommand{\R}{\mathbb{R}}

\newcommand{\Z}{\mathbb{Z}}

\newcommand{\cG}{\mathcal{G}}

\renewcommand{\1}{\mathds{1}}

\renewcommand{\epsilon}{\varepsilon}
\renewcommand{\phi}{\varphi}

\DeclareMathOperator{\cov}{Cov}

\DeclareMathOperator{\off}{Off}

\title{On the proof of the Kozma-Nitzan conjecture}
\author{Jean-Baptiste Gouéré\footnote{Institut Denis Poisson, UMR-CNRS 7013, Université de Tours}}
\date{}

\begin{document}

\maketitle

\begin{abstract}
	In this short note, we give an exposition of the recent proof of the Kozma-Nitzan conjecture provided by OpenAI. 
\end{abstract}

\section{Setting and main results}

\paragraph{Setting.} We consider a finite graph $G$ with vertex set $V$.
Each edge $e$ is open with probability $0<p_e<1$ independently of the others.
We study the connectivity properties of the subgraph consisting of the open edges.
For $a, b \in V$, we denote by $\{a \leftrightarrow b\}$ the event "there is a path of open edges connecting $a$ and $b$" and by $\{a \not\leftrightarrow b\}$
the complementary event.
For $A, B \subset V$, we write $\{A \leftrightarrow B\}$ for the union over $a \in A$ and $b \in B$ of the events $\{a \leftrightarrow b\}$.
We use other similar notation whose meaning we hope is clear.
For $a \in V$, we denote its cluster by $C_a$.
This is the connected subgraph of $G$ whose vertices are those connected to $a$ by open paths and whose edges are all the open edges between these vertices.
In particular it always contain $a$.
For $A \subset V$ we write $C_A$ for the union of all $C_a$ for $a \in A$ (this is not necessarily connected).
When considering functions of clusters, we use the term cluster to mean any nonempty connected subgraph of $G$, without reference to a particular configuration of open and closed edges.
We will occasionally identify a cluster with its vertex set; for example, if $C$ is a cluster and $A \subset V$, we write $C \cap A \neq \emptyset$ to mean that $C$ has a vertex in $A$.

\paragraph{The conjecture of Kozma and Nitzan.} Our aim is to present the proof of the following result given in \cite{OpenAI_Z3}.

\begin{theorem} \label{t:conj_kozma_nitzan}
	Let $A$ be a nonempty subset of $V$.
	Let $s$ and $t$ be two vertices of $V$.
	Then
	\[
		\P[s \leftrightarrow A ; s \leftrightarrow t] \ge \P[s \leftrightarrow A] \min_{a \in A} \P[a \leftrightarrow t].
	\]
\end{theorem}

This is Conjecture 1 of \cite{Kozma_Nitzan}.
In \cite{Kozma_Nitzan}, Kozma and Nitzan showed that the above statement implies the absence of percolation at criticality for Bernoulli bond percolation on $\Z^d$, $d \ge 2$.
This was previously an important conjecture only solved in dimension $2$ \cite{Harris,Kesten} and in dimension $d \ge 11$ \cite{Fitzner_Hofstad}.
Kozma and Nitzan also provide variants of their conjecture, still implying non percolation at criticality, and they establish some of them in various special cases.
In \cite{Flammant}, Flammant introduces yet another variant, which also implies non percolation at criticality, and he proves it for planar graphs where $s$ and $A$ are on the boundary of the same face.

Theorem \ref{t:conj_kozma_nitzan} extends (by continuity) to the case where the $p_e$ take values in $[0,1]$.
It is also stated in \cite{OpenAI_Z3} for hypergraphs, without any change to the proof.
In this expository article, we have chosen to remain in the most elementary setting.

Theorem \ref{t:conj_kozma_nitzan} was proved shortly beforehand in \cite{Anthropic}.
We refer to Blanc-Renaudie's article \cite{Blanc-Renaudie} for a human exposition of the proof of \cite{Anthropic} and further references.
The proofs in \cite{Anthropic} and \cite{OpenAI_Z3} are very similar and rely on the same ideas.
However, certain aspects of the proof are formalized sufficiently differently that a human exposition does not seem entirely superfluous to us.

Finally, let us mention the claimed proof of the absence of percolation at criticality 
for Bernoulli bond percolation on infinite locally finite quasi-transitive graphs with critical probability less than one
\cite{OpenAI_quasitransitive}. We do not consider the content of \cite{OpenAI_quasitransitive} in these notes. 

\paragraph{A more general correlation result.} 
Theorem \ref{t:conj_kozma_nitzan} is proved as an elementary consequence of Theorem \ref{t:general}.
Before stating it, we need to introduce some notation.
Let $x_1, \dots, x_n$ be a list of distinct vertices of $G$. 
We set 
\[
X_r = \{x_1,\dots,x_r\}
\]
for $r \in \{1,\dots,n\}$ and $X_0=\emptyset$.
For each $j \in \{1,\dots,n\}$, we denote by $\bm{h_j}$ the column vector whose $i$-th entry, for $i \in \{1,\dots,n\}$, is given by
\[
	\P[x_i \leftrightarrow x_j | x_j \not\leftrightarrow X_{j-1}].
\]
Finally, we denote by $H$ the matrix in $GL_n(\R)$ whose columns are the $\bm{h_j}$.
In other words, for all $i,j$, we have
\[
	H_{i,j} = \P[x_i \leftrightarrow x_j | x_j \not\leftrightarrow X_{j-1}].
\]
This is a lower triangular matrix with diagonal entries equal to $1$. In particular, it is invertible.
Set 
\[
	\Gamma = I-H^{-1}.
\]
The matrix $\Gamma$ is lower triangular with zero diagonal.

\begin{theorem} \label{t:general}  
	Let $F$ be a nonnegative nondecreasing real-valued function from clusters of $G$.
	Then, for every $i,r \in \{1,\dots,n\}$,
	\begin{equation}\label{e:general_1}
		\E\big[F(C_{x_r}) ; x_i \leftrightarrow x_r \big| x_r \not\leftrightarrow X_{r-1}\big] \ge \sum_{j<i} \Gamma_{i,j} \E\big[F(C_{x_r}) ; x_j \leftrightarrow x_r \big| x_r \not\leftrightarrow X_{r-1} \big]
	\end{equation}
	and
	\begin{equation}\label{e:general_2}
		\E[F(C_{x_i})] \ge \sum_{j<i} \Gamma_{i,j} \E[F(C_{x_j})].
	\end{equation}
	Moreover, the entries of $\Gamma$ are nonnegative and, for every $i \in \{1,\dots,n\}$,
	\[
		\sum_{j<i} \Gamma_{i,j} = \P[x_i \leftrightarrow X_{i-1}].
	\]
\end{theorem}

The first part (that is, relation \eqref{e:general_1} for every sequence $(x_1,\dots,x_n)$, all $F$ and all $i,r$) is precisely Theorem 3 of \cite{Blanc-Renaudie}, stated in a different form.
We elaborate on this connection in the comments following Proposition \ref{p:cov_cone}. 
The rest of Theorem \ref{t:general} follows from \eqref{e:general_1}.
Theorem \ref{t:conj_kozma_nitzan} is obtained by combining \eqref{e:general_2} and the properties of $\Gamma$.
Thus, all the results stem from \eqref{e:general_1}.
The situation is similar in \cite{Blanc-Renaudie}: all its results follow from its Theorem 3.

\begin{proof}[Proof of Theorem \ref{t:conj_kozma_nitzan} from Theorem \ref{t:general}]
	If $s \in A$, the result is trivial.
	In what follows, we assume that this is not the case.
	Write $A \cup \{s\} = \{x_1,\dots,x_n\}$, where the $x_i$ are distinct and $x_n=s$.
	Let $F$ be the real-valued function from clusters of $G$  defined by 
	\[
		F(C) = \1_{C \cap A \neq \emptyset ; t \in C}.
	\]
	By \eqref{e:general_2} with $i=n$, we have
	\[
		\E[F(C_{x_n})] \ge \sum_{j<n} \Gamma_{n,j} \E[F(C_{x_j})]
	\]
	and hence, writing out the expectations explicitly, 
	\[
		\P[s \leftrightarrow A ; s \leftrightarrow t] \ge \sum_{j<n} \Gamma_{n,j} \P[x_j \leftrightarrow t].
	\]
	Using the nonnegativity of $\Gamma$ and the values of its row sums (see Theorem \ref{t:general}), we obtain
	successively
	\[
		\P[s \leftrightarrow A ; s \leftrightarrow t] 
		\ge  \sum_{j<n} \Gamma_{n,j} \min_{a \in A} \P[a \leftrightarrow t]
		= \P[s \leftrightarrow A] \min_{a \in A} \P[a \leftrightarrow t].
	\]
	This concludes the proof.
\end{proof}

\paragraph{Status of these notes.} 
This is not a research article. 
There is no claim of originality.
This is simply an exposition of a proof provided by an AI.

\section{Proof of Theorem \ref{t:general}}

For any $r \in \{1,\dots,n\}$ we consider the disconnection event
\[
	D_r = \{x_r \not\leftrightarrow X_{r-1}\},
\]
which has positive probability, 
and denote by $\E_{D_r}$, $\P_{D_r}$ and $\cov_{D_r}$ the associated conditional expectation, probability and covariance.

If $F$ is a real-valued function of clusters and if $r \in \{1,\dots,n\}$, 
we denote by $\bm{h_r^F}$ the column vector whose $i$-th entry, for $i \in \{1,\dots,n\}$, is given by
\[
	\E_{D_r}[F(C_{x_r}) ; x_i \leftrightarrow x_r].
\]
The first $r-1$ entries of this column vector are zero (because of the conditioning), and the next one is $\E_{D_r}[F(C_{x_r})]$.
Let $\alpha_1^F, \dots, \alpha_n^F$ be the coordinates of $\bm{h_r^F}$ in the basis formed by the $\bm{h_j}$. 
Recall that $\bm{h_j}$ denotes the column vector with entries $\P_{D_j}[x_i \leftrightarrow x_j]$,
that the first $j-1$ entries of the column vector $\bm{h_j}$ are zero, and that the next one equals $1$.
In view of these observations, we have
\[
	\alpha_1^F = \dots = \alpha_{r-1}^F = 0 \text{ and } \alpha_r^F = \E_{D_r}[F(C_{x_r})].
\]
Thus 
\[
	 \bm{h_r^F} = \E_{D_r}\big[ F(C_{x_r})\big] \bm{h_r} + \sum_{j=r+1}^n \alpha^F_j \bm{h_j}.
\]
Note that the $i$-th entry of 
\[
	\bm{h_r^F} - \E_{D_r}\big[ F(C_{x_r})\big] \bm{h_r} = \bm{h_r^{F-\E_{D_r}[ F(C_{x_r})]}}
\]
is the covariance under $\P_{D_r}$ of $F(C_{x_r})$ with $\1_{x_i \leftrightarrow x_r}$.
In the following proposition, we show that	all the $\alpha_j^F$ are nonnegative whenever $F$ is nonnegative and nondecreasing.
Thanks to the above remarks, this is equivalent to saying that all the $\alpha_j^F$, $j>r$, are nonnegative whenever $F$ is nondecreasing.
Let us give a formula for the coefficient $\alpha_j^F$.
Recall that the matrix $\Gamma$ is defined by 
\[
	\Gamma = I-H^{-1}.
\]
Let $\bm{\alpha^F}$ denote the column vector with entries $\alpha_j^F$.
By definition, we have $\bm{h_r^F} = H \bm{\alpha^F}$ and hence $\bm{\alpha^F} = H^{-1}\bm{h_r^F}$, giving
\begin{equation} \label{e:formule_alpha}
	\bm{\alpha^F} = \bm{h_r^F} - \Gamma \bm{h_r^F}.
\end{equation}

\begin{prop} \label{p:cov_cone}
	Let $F$ be a nonnegative nondecreasing real-valued function of clusters and let $r \in \{1,\dots,n\}$. 
	Then the coordinates $\alpha_j^F$ of $\bm{h_r^F}$ in the basis formed by the $\bm{h_j}$ are nonnegative.
	In other words, $\bm{h_r^F}   \ge \Gamma \bm{h_r^F}$, that is, for every $i \in \{1,\dots,n\}$,
	\[
		\E_{D_r}[F(C_{x_r}) ; x_i \leftrightarrow x_r] \ge \sum_{j<i} \Gamma_{i,j} \E_{D_r}[F(C_{x_r}) ; x_j \leftrightarrow x_r].
	\]
\end{prop}

We prove this proposition in Section \ref{s:preuve_prop}.
Let us just note here that the second part is an immediate rewriting of the first part, using \eqref{e:formule_alpha} and 
the known fact that $\Gamma$ is lower triangular with zero diagonal.
The key fact is thus the nonnegativity of the coordinates $\alpha_j^F$.
This is actually equivalent to Theorem 3 in \cite{Blanc-Renaudie}, which is also the key in \cite{Blanc-Renaudie}.
Let us elaborate on this point as the notation is quite different.
As explained above Proposition \ref{p:cov_cone}, the result can equivalently be stated as follows.
For any nondecreasing $F$, any sequence of distinct vertices $(x_1,\dots,x_n)$ and any $r \in \{1,\dots,n\}$, 
the unique $\alpha_{r+1}^F, \dots, \alpha_n^F$ satisfying
\[
	\bm{h_r^{F-\E_{D_r}[ F(C_{x_r})]}} = \sum_{j=r+1}^n \alpha^F_j \bm{h_j}
\]
are nonnegative.
As the matrix $H$ with columns $\bm{h_j}$ is lower-triangular with diagonal entries equal to $1$, 
we can easily solve the equation step by step by  successively computing $\alpha^F_{r+1}, \dots, \alpha^F_n$. 
It appears that each step involves one of the operators $\pi$ of \cite{Blanc-Renaudie}.
In the end, for $j>r$, the coefficient $\alpha_j^F$ satisfies, with the notation of \cite{Blanc-Renaudie},
\[
	\alpha_j^F = \pi_{X_r,(x_{r+1},\dots,x_{j-1})} ( \kappa_{x_r,X_{r-1}}F)(x_j).
\]
We thus recover the statement of Theorem 3 in \cite{Blanc-Renaudie}.

In the remainder of this section, we show how Proposition \ref{p:cov_cone} implies Theorem~\ref{t:general}.
First, Proposition \ref{p:cov_cone} immediately gives \eqref{e:general_1}.
The inequality \eqref{e:general_2} is easily deduced.
We encapsulate it in the following lemma.

\begin{lemma} \label{l:inegalite_2}  
	Let $F$ be a nonnegative nondecreasing read-valued function from clusters of $G$.
	Then, for every $i \in \{1,\dots,n\}$,
	\[
		\E[F(C_{x_i})] \ge \sum_{j<i} \Gamma_{i,j} \E[F(C_{x_j})].
	\]
\end{lemma}
\begin{proof}
	For each index $i$, decomposing according to the first point $x_r$ in the cluster of $x_i$, we have
	\begin{align*}
		\E[F(C_{x_i})] 
			& = \sum_{r=1}^n \E\big[ F(C_{x_i}) ; x_i \leftrightarrow x_r ; x_r \not\leftrightarrow X_{r-1}\big] \\
			& = \sum_{r=1}^n \E\big[ F(C_{x_r}) ; x_i \leftrightarrow x_r ; x_r \not\leftrightarrow X_{r-1}\big] \\
			& = \sum_{r=1}^n \P[D_r]\E_{D_r}\big[ F(C_{x_r}) ; x_i \leftrightarrow x_r \big].
	\end{align*}
	The result then follows from Proposition \ref{p:cov_cone}.
\end{proof}

The following lemma constitutes the last part of Theorem \ref{t:general}.
The first assertion is elementary. 
The second assertion again follows from Proposition \ref{p:cov_cone}.

\begin{lemma} \label{l:gamma}
	\begin{enumerate}
		\item The sum of the entries in the $i$-th row of $\Gamma$ equals $\P[x_i \leftrightarrow X_{i-1}]$.		
		\item The entries of $\Gamma$ are nonnegative.
	\end{enumerate}
\end{lemma}

\begin{proof}
	Let us prove the first assertion. For every index $i$, we have 
	\[
		1 = \sum_{j=1}^n \P[x_i \leftrightarrow x_j ; x_j \not\leftrightarrow X_{j-1}] = \sum_{j=1}^n  \P_{D_j}[x_i \leftrightarrow x_j]\P[D_j].
	\]
	In matrix form, this reads $\bm 1 = H \bm P$, where $\bm 1$ denotes the column vector whose entries are all $1$ 
	and $\bm P$ is the column vector with entries $\P[D_j]=\P[x_j \not\leftrightarrow X_{j-1}]$.
	Multiplying by $H^{-1} = I - \Gamma$ gives $\bm P = (I-\Gamma)\bm 1$ and then $\Gamma \bm 1 = \bm 1- \bm P$.
	This proves the first assertion.
	
	We now show that the entries of $\Gamma$ are nonnegative.
	Fix an index $j$ and consider the real-valued function of clusters defined by $F(C) = \1_{C \cap \{x_{j+1},\dots,x_n\} \neq \emptyset}$ so that
	\[
		F(C_{x_j}) = \1_{x_j \leftrightarrow \{x_{j+1},\dots,x_n\}}.
	\]
	Then, 
	\[
		\bm{h_j^F} = \bm{h_j} - q \bm{e_j}
	\]
	where
	\[
		q = \P_{D_j}\big[x_j \not\leftrightarrow \{x_{j+1},\dots,x_n\}\big]>0
	\]
	and where $\bm{e_j}$ is the $j$-th vector of the canonical basis.
	By Proposition \ref{p:cov_cone}, recalling that $I-\Gamma=H^{-1}$, we obtain
	\[
		\bm{0} \le \bm{h_j^F} - \Gamma \bm{h_j^F} = (I-\Gamma)\bm{h_j} -q \bm{e_j} + \Gamma q\bm{e_j} = (1-q)\bm{e_j} + \Gamma q\bm{e_j}.
	\]
	For any index $i>j$, considering the $i$-th entry, we get $\Gamma_{i,j} \ge 0$.
	For any index $i \le j$ we already know $\Gamma_{i,j}=0$.
\end{proof}

\section{Proof of Proposition \ref{p:cov_cone}}

\label{s:preuve_prop}

For all cluster or union of clusters $C$, we define $\P_{\off C}$ as the probability under which all edges incident to at least one vertex of $C$ are closed,
and all other edges retain their original product law.
We denote by $\E_{\off C}$ and $\cov_{\off C}$ the associated expectation and covariance.
For example, for any real-valued function $F$ defined on clusters,
\begin{equation}\label{e:ex_off}
	\E_{D_r}\big[ F(C_{x_r}) | C_{X_{r-1}}\big] = \E_{\off C_{X_{r-1}}}\big[F(C_{x_r})].
\end{equation}
This follows from the fact that, by independence, exploring the cluster of $X_{r-1}$ gives no information on unrevealed edges.
Therefore, under $\P_{D_r}$, conditional on $C_{X_{r-1}}$, the states of edges which are not incident to $C_{X_{r-1}}$ are independent with their initial Bernoulli parameter, and the edges with only one vertex in $C_{X_{r-1}}$ are closed.
But the states of those edges have the same distribution under $\P_{\off C_{X_{r-1}}}$, and $C_{x_r}$ depends only on the state of those edges.

\begin{lemma} \label{l:esp_cov_cond}
	Let $r \in \{1,\dots,n\}$ and let $F$ be a nonnegative nondecreasing function of clusters.
	Then, for every $i \in \{1,\dots,n\}$,
	\begin{equation} \label{e:esp_cov_cond1}
		\E_{D_r}\Big[\cov_{D_r}\big(\1_{x_i \leftrightarrow x_r}, F(C_{x_r}) \big| C_{X_{r-1}}\big)\Big]  
		= \sum_{j=r+1}^n \E_{D_j} \big[\1_{x_i \leftrightarrow x_j} F_j(C_{x_j})  \big]
	\end{equation}
	where, for each $j>r$, $F_j$ is the function of clusters defined by 
	\begin{align}
		F_j(C) 
		& = \frac{\P[x_j \not\leftrightarrow X_{j-1}]}{\P[x_r \not\leftrightarrow X_{r-1}]}  
		\E_{\off C}\Big[\1_{x_r \not\leftrightarrow X_{r-1}} \big(\E_{\off C_{X_{r-1}}}\big[F(C_{x_r})\big] - F(C_{x_r}) \big) \Big].  \label{e:esp_cov_cond3}
	\end{align}
	Moreover, each $F_j$ is nonnegative and nondecreasing.
\end{lemma}
\begin{proof}
	For all indices $i,j$, let $A(i,j)$ denote the event "$x_j$ is the first point in $C_{x_i}$".
	More explicitly,
	\[
		A(i,j) = \{x_i \leftrightarrow x_j ; x_i \not\leftrightarrow X_{j-1}\} = \{x_i \leftrightarrow x_j ; x_j \not\leftrightarrow X_{j-1}\}.
	\]
	Observe that 
	\begin{equation}\label{e:depliage}
		\cov_{D_r}\big(\1_{x_i \leftrightarrow x_r} , F(C_{x_r}) \big| C_{X_{r-1}} \big)
		=
		- \sum_{j>r} \cov_{D_r}\big(\1_{A(i,j)}, F(C_{x_r}) \big| C_{X_{r-1}} \big).
	\end{equation}
	Let us justify this. Recall that, under $\P_{D_r}$, $x_r$ is not connected to any point of $X_{r-1}$.
	If $x_i \in C_{X_{r-1}}$, then the indicators are constant and hence the covariances vanish	, which gives the equality.
	If $x_i \not\in C_{X_{r-1}}$, then the above equality follows from 
	\[
		\1 = \sum_{j \ge r} \1_{A(i,j)} = \1_{x_i \leftrightarrow x_r} + \sum_{j>r} \1_{A(i,j)}.
	\]
	Now fix $j>r$. We have
	\begin{align*}
		\E_{D_r}\Big[\cov_{D_r}\big(\1_{A(i,j)}, F(C_{x_r}) \big| C_{X_{r-1}} \big)\Big]
		& = \E_{D_r}\Big[\1_{A(i,j)}\big(F(C_{x_r}) - \E_{D_r}\big[F(C_{x_r}) \big| C_{X_{r-1}}\big]\big)\Big].
	\end{align*}
	We now pass from an expectation under $\P_{D_r}$ to an expectation under $\P_{D_j}$ using
	\[
		\{x_r \not\leftrightarrow X_{r-1}\} \cap A(i,j) = \{x_r \not\leftrightarrow X_{r-1}\} \cap \{x_j \not\leftrightarrow X_{j-1}\} \cap \{x_i \leftrightarrow x_j\}.
	\]
	Setting
	\[
		\rho_j =  \frac{\P[x_j \not\leftrightarrow X_{j-1}]}{\P[x_r \not\leftrightarrow X_{r-1}]},
	\]
	we obtain the identity
	\begin{align*}
		\E_{D_r}\Big[\cov_{D_r}\big(\1_{A(i,j)}, F(C_{x_r}) \big| C_{X_{r-1}} \big)\Big]
		& = \rho_j \E_{D_j}\Big[\1_{x_i \leftrightarrow x_j}\1_{x_r \not\leftrightarrow X_{r-1}} \big(F(C_{x_r}) - \E_{D_r}\big[F(C_{x_r}) \big| C_{X_{r-1}}\big]\big)\Big].
	\end{align*}
	We now condition on $C_{x_j}$ to obtain, using the definition of $F_j$ in \eqref{e:esp_cov_cond3},
	\begin{align*}
		& - \E_{D_r}\Big[\cov_{D_r}\big(\1_{A(i,j)}, F(C_{x_r}) \big| C_{X_{r-1}} \big)\Big] \\
		& = \E_{D_j}\Bigg[\1_{x_i \leftrightarrow x_j} \rho_j \E_{D_j}\Big[ \1_{x_r \not\leftrightarrow X_{r-1}} \big(\E_{D_r}\big[F(C_{x_r}) \big| C_{X_{r-1}}\big]- F(C_{x_r})\big) \Big| C_{x_j}\Big]\Bigg] \\
		& =  \E_{D_j}\big[\1_{x_i \leftrightarrow x_j} F_j(C_{x_j})\big].
	\end{align*}
	Combining the preceding equality with \eqref{e:depliage} yields \eqref{e:esp_cov_cond1}.
	
	
	Finally, the nonnegativity and monotonicity of $F_j$ are stated in Lemma 9 of \cite{Blanc-Renaudie}.
	Let us give a proof for completeness.
	Let $\omega$ and $\omega'$ be two independent copies of  the unconditioned environment.
	Using notation that is hopefully self-explanatory, we can rewrite the definition of $F_j$ as follows:
	\begin{align*}
		F_j(C) 
		& = \rho_j\E\Bigg[\big(\1-\1_{C_{X_{r-1}}(\omega \off C)}(x_r)\big) \Big(F\big(C_{x_r}(\omega' \off C_{X_{r-1}}(\omega \off C))\big) - F(C_{x_r}(\omega \off C))\Big)\Bigg].
	\end{align*}
	But exploring $C_{X_{r-1}}(\omega \off C)$ does not reveal the state of any edge off $C_{X_{r-1}}(\omega \off C)$.
	Therefore we can actually replace $\omega'$ by $\omega$, that is,
	\begin{align*}
		F_j(C) 
		& = \rho_j \E\Bigg[\big(\1-\1_{C_{X_{r-1}}(\omega \off C)}(x_r)\big) \Big(F\big(C_{x_r}(\omega \off C_{X_{r-1}}(\omega \off C))\big) - F(C_{x_r}(\omega \off C))\Big)\Bigg].
	\end{align*}
	In other words, we have coupled the randomness of the two nested expected values by using the same environment.
	It can now be checked that the integrand is nonnegative and nondecreasing in $C$ for each fixed $\omega$, which implies the result.
	
	Consider first nonnegativity. 
	Assume $\1-\1_{C_{X_{r-1}}(\omega \off C)}(x_r)=1$ (otherwise nonnegativity is trivial).
	This means that $x_r$ is not connected to $X_{r-1}$ in the environment $\omega \off C$.
	As $F$ is nondecreasing, it now suffices to prove 
	\[
		C_{x_r}(\omega \off C) \subset C_{x_r}(\omega \off C_{X_{r-1}}(\omega \off C)).
	\]
	Let $\pi$ be a path of open edges in $\omega \off C$ originating from $x_r$.
	It suffices to check that no vertex of $\pi$ is in $C_{X_{r-1}}(\omega \off C)$.
	But if this were the case, then $x_r$ and $X_{r-1}$ would be connected in the environment $\omega \off C$.
	
	Let us prove that the integrand is nondecreasing in $C$.
	As we know that the integrand is nonnegative and as $\1-\1_{C_{X_{r-1}}(\omega \off C)}(x_r)$ is nondecreasing in $C$,
	 it suffices to show that
	\[
	 	F\big(C_{x_r}(\omega \off C_{X_{r-1}}(\omega \off C))\big) - F(C_{x_r}(\omega \off C))
	\]
	is nondecreasing in $C$. 
	But increasing $C$ decreases $C_{X_{r-1}}(\omega \off C)$ and thus increases $C_{x_r}(\omega \off C_{X_{r-1}}(\omega \off C))$ and therefore increases the first term.
	The second term is handled similarly.
\end{proof}

Let $r \in \{1,\dots,n\}$.
Denote by $\hat\Omega$ the set of all edge configurations satisfying $D_r$.
We consider the discrete-time Markov chain on $\hat\Omega$ with the following transition rule.
Given the configuration $\widehat\omega^t$ at time $t$, obtain $\widehat\omega^{t+1}$ in two half-steps:
\begin{enumerate}
	\item Independently resample all edges with no endpoint in $C_{x_r}(\widehat\omega^t)$, each according to its original Bernoulli law, leaving all other edges unchanged.
	Denote the resulting configuration by $\widehat\omega^{t+\frac12}$.
	\item Independently resample all edges with no endpoint in $C_{X_{r-1}}(\widehat\omega^{t+\frac12})$,	each according to its original Bernoulli law, 	leaving all other edges unchanged.
	Denote the resulting configuration by $\widehat\omega^{t+1}$.
\end{enumerate}
This  is the Markov chain introduced by van den Berg, Häggström and Kahn \cite{BHK} to prove the BHK correlation inequality
(see their Section 2.1 where they work in the more general setting of random-cluster measures).
See also the review article by Häggström \cite{Haggstrom_cooking}.

The process $C_{x_r}(\widehat\omega^t)$ is still a Markov chain, as the knowledge of $C_{x_r}(\widehat\omega^t)$ is sufficient to 
perform the above two half-steps. 
This is a Markov chain on the state space of clusters containing $x_r$ and not touching $X_{r-1}$.
Let us call them $r$-admissible clusters.
Denote by $\cG_r$ the associated operator on real-valued functions $F$ of clusters. 
It satisfies
\[
	\cG_rF(C) =  \E_{\off C}\Big[ \E_{\off {C_{X_{r-1}}}} \big[ F(C_{x_r}) \big] \Big]
\]
for all $r$-admissible clusters $C$. 
Actually we only need the above operator.
Nevertheless we started the presentation with the Markov process on $\hat\Omega$ as we estimate it enlighten what happens.

\begin{lemma} \label{l:cov_esp_cond}
	Let $r \in \{1,\dots,n\}$ and let $F$ and $G$ be real-valued functions of clusters.
	Then, 
	\begin{equation} \label{e:cov_esp_cond1}
		\cov_{D_r}\Big(\E_{D_r}\big[F(C_{x_r}) \big|C_{X_{r-1}}\big] , \E_{D_r}\big[ G(C_{x_r}) \big| C_{X_{r-1}}\big]\Big)
		= \cov_{D_r}\Big(F(C_{x_r}), \cG_r G(C_{x_r})\Big).
	\end{equation}
	Moreover, for any $r$-admissible cluster $C$,
	the $t$-th iterate $\cG^t_r(F)(C)$ converges to $\E_{D_r}[F(C_{x_r})]$ as $t$ tends to infinity.
	Finally, if $F$ is nondecreasing and nonnegative, then $\cG_rF$ is also nondecreasing and nonnegative.
\end{lemma}

\begin{proof}
	Identity \eqref{e:cov_esp_cond1} follows from the following computation:
	\begin{align*}
		\E_{D_r}\big[F(C_{x_r}) \cG_r G(C_{x_r})\big]
		& = \E_{D_r}\bigg[F(C_{x_r}) \E_{\off C_{x_r}}\Big[ \E_{\off {C_{X_{r-1}}}} \big[ G(C_{x_r}) \big] \Big] \bigg] \\
		& = \E_{D_r}\bigg[F(C_{x_r}) \E_{D_r}\Big[ \E_{\off {C_{X_{r-1}}}} \big[ G(C_{x_r}) \big]\Big| C_{x_r} \Big] \bigg] \\
		& = \E_{D_r}\bigg[F(C_{x_r})  \E_{\off {C_{X_{r-1}}}} \big[ G(C_{x_r}) \big]\bigg] \\	
		& = \E_{D_r}\bigg[F(C_{x_r}) \E_{D_r}\Big[G(C_{x_r}) | C_{X_{r-1}}\Big]\bigg] \\	
		& = \E_{D_r}\bigg[\E_{D_r}\Big[F(C_{x_r}) | C_{X_{r-1}}\Big] \E_{D_r}\Big[G(C_{x_r}) | C_{X_{r-1}}\Big]\bigg].
	\end{align*}
	We used properties such as \eqref{e:ex_off} in the second and fourth equalities.
	We used basic properties of conditional expectation in the third and fifth equalities.
	Taking $F=1$ also gives 
	\[
		\E_{D_r}\big[\cG_r G(C_{x_r})\big] = \E_{D_r}\big[ G(C_{x_r})\big].
	\]
	This yields \eqref{e:cov_esp_cond1}. 
	
	The previous equality also implies that the distribution of $C_{x_r}$ under $\P_{D_r}$ is invariant for the Markov chain
	associated with $\cG_r$. \footnote{One can also see it from the fact that $\P_{D_r}$ is invariant for the Markov chain on $\hat\Omega$
	introduced above the lemma.} But the Markov chain started from any $r$-admissible cluster
	converges in law to its stationary distribution 
	(for example because it is possible to go in one step from any $r$-admissible cluster to the trivial cluster $\{x_r\}$ 
	and from this one to any $r$-admissible cluster). This yields the second assertion of the Lemma.
	
	The last point is elementary.
\end{proof}

We are now ready to prove Proposition \ref{p:cov_cone}.
Recall that the second part follows from the first part (see the remark below Proposition \ref{p:cov_cone}).
The proof proceeds by downward induction on $r$.
The result is immediate for $r=n$ (see the remarks preceding Proposition \ref{p:cov_cone}).
Let $r<n$ be such that the result has been proved for $r' \in \{r+1,\dots,n\}$.
Let $F$ be a nonnegative nondecreasing function of clusters.
For every $t \ge 0$, $\cG_r^t F$ is nonnegative and nondecreasing by Lemma \ref{l:cov_esp_cond}.
By the law of total covariance, followed by Lemmas \ref{l:esp_cov_cond} and \ref{l:cov_esp_cond}, we have
(below we manipulate functions of the $x_i$'s)
\begin{align*}
	& \cov_{D_r}\big(1_{\cdot \leftrightarrow x_r}, \cG_r^t F(C_{x_r})\big) \\
	& = \E_{D_r}\Big[\cov_{D_r}\big(\1_{\cdot\leftrightarrow x_r}, \cG_r^t F(C_{x_r}) \big| C_{X_{r-1}}\big)\Big] 
		+  \cov_{D_r}\Big(\E_{D_r}\big[\1_{\cdot \leftrightarrow x_r} \big|C_{X_{r-1}}\big] , \E_{D_r}\big[ \cG_r^t F(C_{x_r}) \big| C_{X_{r-1}}\big]\Big) \\
		& = \sum_{j=r+1}^n \E_{D_j} \big[\1_{\cdot \leftrightarrow x_j} (\cG_r^t F)_j(C_{x_j})  \big] + \cov_{D_r}\Big(\1_{\cdot \leftrightarrow x_r}, \cG^{t+1}_r F(C_{x_r})\Big)
\end{align*}
and hence 
\begin{equation}\label{e:pre_tele}
	\cov_{D_r}\big(1_{\cdot \leftrightarrow x_r}, \cG_r^t F(C_{x_r})\big) - \cov_{D_r}\Big(\1_{\cdot \leftrightarrow x_r}, \cG^{t+1}_r F(C_{x_r})\Big)
	= \sum_{j=r+1}^n \E_{D_j} \big[\1_{\cdot \leftrightarrow x_j} (\cG_r^t F)_j(C_{x_j})  \big].
\end{equation}
Consider the term with index $j$ on the right-hand side.
The function $(\cG_r^t F)_j$ is nonnegative and nondecreasing by Lemmas  \ref{l:esp_cov_cond} and \ref{l:cov_esp_cond}.
By the induction hypothesis, the coordinates of the corresponding vector in the base $(\bm{h_1},\dots,\bm{h_n})$ are nonnegative.
Moreover (see the remarks above Proposition \ref{p:cov_cone}) its first $j-1$ coefficients vanish.
Therefore each term on the right-hand side belongs (identifying these functions of the $x_i$'s with column vectors) to the closed cone
\[
	K = \left\{ \sum_{k \ge r} \alpha_k \bm{h_k} \text{ for } \alpha_r, \dots, \alpha_n \ge 0\right\}.
\]
As a consequence, the left-hand side of \eqref{e:pre_tele} belongs to $K$.
Summing (the sum telescopes) we get, for all $t$,
\[
	\cov_{D_r}\big(1_{\cdot \leftrightarrow x_r},  F(C_{x_r})\big) - \cov_{D_r}\Big(\1_{\cdot \leftrightarrow x_r}, \cG^{t}_r F(C_{x_r})\Big) \in K.
\]
But the convergence of $\cG^t_r$ implies the convergence of the second covariance term to $0$.
Therefore
\[
	\cov_{D_r}\big(1_{\cdot \leftrightarrow x_r},  F(C_{x_r})\big) \in K.
\]
Moreover,
\[	
	\E_{D_r}\big[F(C_{x_r})\big] \bm{h_r} \in K
\]
as $F$ is nonnegative. Adding the last two displays yields the result.

\end{document}